\documentclass{amsart}

\usepackage[latin1]{inputenc}   
\usepackage[T1]{fontenc}       
\usepackage{amsmath}
\usepackage{amssymb}
\usepackage{amsthm}

\makeatletter
\@namedef{subjclassname@2010}{%
  \textup{2010} Mathematics Subject Classification}
\makeatother

\usepackage[dvipsnames]{xcolor}

\usepackage{bm}
\usepackage{mathtools}
\usepackage{mathrsfs}
\usepackage{accents} 
\usepackage{hyperref}
\hypersetup{pdfstartview=}

\newcommand{\beq}{\begin{equation}}
\newcommand{\eeq}{\end{equation}}

\newcommand{\Z}{\ensuremath{\mathbb{Z}}}
\newcommand{\R}{\ensuremath{\mathbb{R}}}

\newcommand{\SL}{\ensuremath{\mathrm{SL}}}

\newcommand{\Tr}{\operatorname{Tr}}

\newcommand{\llt}{\ensuremath{\log \log T}}

\renewcommand{\epsilon}{\varepsilon}

\newcommand{\A}{\ensuremath{\mathcal{A}}}
\renewcommand{\P}{\ensuremath{\mathcal{P}}}

\newcommand{\bP}{\ensuremath{\mathbb{P}}}
\newcommand{\F}{\ensuremath{\mathbb{F}}}

\title{An analogue of the Erd\H{o}s--Kac theorem for the special linear group over the integers}
\author{Daniel El-Baz}
\address{School of Mathematical Sciences \\ Tel Aviv University \\ Tel Aviv \\ Israel}
\address{Max Planck Institute for Mathematics \\ Vivatsgasse 7 \\ 53111 Bonn \\ Germany}

\email{danielelbaz88@gmail.com}

\date{}

\begin{document}

\begin{abstract}
We investigate the number of prime factors of individual entries for matrices in the special linear group over the integers. We show that, when properly normalised, it satisfies a central limit theorem of Erd\H{o}s--Kac-type. To do so, we employ a sieve-theoretic set-up due to Granville and Soundararajan. We also make use of an estimate coming from homogeneous dynamics due to Gorodnik and Nevo. 
\end{abstract}



\maketitle

\newtheorem{defn}{Definition}[section]
\newtheorem{claim}{Claim}[section]
\providecommand*{\claimautorefname}{Claim}
\newtheorem{thm}{Theorem}[section]
\providecommand*{\thmautorefname}{Theorem}
\newtheorem{lemma}{Lemma}[section]
\providecommand*{\lemmaautorefname}{Lemma}
\newtheorem{cor}{Corollary}[section]
\providecommand*{\corautorefname}{Corollary}
\newtheorem{rmk}{Remark}[section]
\providecommand*{\rmkautorefname}{Remark}

\section{Introduction} \label{intro}

The celebrated Erd\H{o}s--Kac theorem is a central limit theorem for the number of (distinct) prime factors $\omega$ of a ``random'' integer, in the following sense: 
for every $x \in \R$, we have  
\[\lim_{n \to +\infty} \frac 1n \# \left\{1 \le m  \le n \, : \, \frac {\omega(m) - \log \log n}{\sqrt{\log \log n}} \le x \right\} = \frac 1{\sqrt{2 \pi}} \int_{-\infty}^x e^{- \frac{t^2}{2}} dt. \]

There is an abundant number of results of this type for the number of prime factors of various sequences of integers: shifted primes \cite{Halberstam1955}, values of integer polynomials \cite{Halberstam1956} and friable numbers \cite{Hensley1985, Alladi1987, Hildebrand1987} to cite but a few examples. 

In this note, we study the number of prime factors of the entries in integer matrices of unit determinant.
More precisely, we define, for $n \ge 2$ and $T > 0$, \beq V_T(\Z) = \{g \in \SL_n(\Z) : \| g \| \le T \}, \eeq
where for $g \in \SL_n(\Z)$, $\| g \| = \sqrt{\Tr(g ^{\mathrm{t}} g)}$ is the Frobenius norm.

For an integer $n \ge 1$, we let $\omega(n)$ be the number of distinct prime factors of $n$. We extend $\omega$ to $\Z$ by defining $\omega(0)=0$ and for $n \ge 1, \, \omega(-n) = \omega(n)$.
Our main theorem can now be stated.

\begin{thm} \label{eksln} For every $n \ge 2$, every $x \in \R$ and for each pair $(i,j) \in \{1, \ldots, n\}^2$, we have
\beq 
\lim_{T \to +\infty} 
\frac 1{\#V_T(\Z)}  \#\left \{g \in V_T(\Z) \, : \, \frac{\omega(g_{i,j})-\llt}{\sqrt{\llt}} \le x \right \}
= \frac 1{\sqrt{2 \pi}} \int_{-\infty}^x e^{- \frac{t^2}2} dt.  \eeq
\end{thm}

Our proof uses the sieve-theoretic framework unravelled by Granville and Soundararajan in 2006 \cite{GranvilleSound}, which we recall in \autoref{GSsieve} so as to be self-contained.
To obtain the necessary estimates to feed into the sieve, we apply a deep result of Gorodnik and Nevo from 2010 \cite{GorodNevo}. 

We note that since this effective congruence counting is the key input, one can also apply such arguments to thin linear groups, thanks to the work of Bourgain, Gamburd and Sarnak \cite{BGS2010}. Indeed, this was done in the case of Apollonian circle packings with integral curvatures by Djankovi\'c in \cite{Djankovic2011}.

\textbf{Acknowledgements:} I would like to thank Efthymios Sofos for asking the question answered in this paper. I am grateful to Zeev Rudnick for several helpful conversations and encouraging me to write up this short note. I am also grateful to Carlo Pagano for the comments he made on a previous version of this draft that led to an improved presentation. 
This research was supported by the  
European Research Council under the European Union's Seventh
Framework Programme (FP7/2007-2013) / ERC grant agreement
n$^{\text{o}}$ 320755.

\section{Counting with congruences}
We need the following special case of~\cite[Corollary 5.2]{GorodNevo}.

For $n \ge 2$ and a positive integer $Q$, define the following --- principal congruence --- subgroup of $\SL_n(\Z)$:
\beq \Gamma(Q) = \{ g \in \SL_n(\Z) \, : \, g \equiv I_n \pmod Q  \}. \eeq

\begin{claim}[Gorodnik--Nevo] \label{GoroNevo} For every $n \ge 2$, there exists $\delta>0$ such that for every $M \in \SL_n(\Z)$ and every integer $Q \ge 1$,
\beq \# \{A \in \Gamma(Q) M : \|A\| \le T \} = \frac {\#V_T(\Z)}{[\SL_n(\Z):\Gamma(Q)]} + O(\#V_T(\Z)^{1-\delta}), \eeq
where the implied constant is independent of $Q$.
\end{claim}

In fact, we shall require the following consequence.

\begin{cor} \label{key_input} For every $n \ge 2$, there exists $\delta > 0$ such that for every square-free integer $q$ and every $(i,j) \in \{1, \ldots, n \}^2$,
\beq 
\#\{ g \in V_T(\Z) \, : \, q \mid g_{i,j}  \} = \prod_{\substack{p \in \bP \\ p \mid q}} \frac {p^{n-1}-1}{p^n-1} \#V_T(\Z) + O\left(q^{n^2-2} \#V_T(\Z)^{1-\delta} \right)
\eeq
\end{cor}


\begin{proof}
Fix $n \ge 2$ and $(i,j) \in \{1, \ldots, n \}^2$.

The group $\Gamma(q)$ acts by left multiplication on the set $\{ g \in \SL_n(\Z) \, : \, q \mid g_{i,j}  \}$, which we may therefore view as a disjoint union of finitely (a priori, at most $q^{n^2-1}$)
$\Gamma(q)$-orbits.
We restrict the orbits to elements having norm at most $T$ and note that each orbit has the same number of points, given by the formula in \autoref{GoroNevo}: 
\beq \frac {\#V_T(\Z)}{[\SL_n(\Z):\Gamma(q)]} + O(\#V_T(\Z)^{1-\delta}). \eeq
First, $\SL_n(\Z)/\Gamma(q) \cong \SL_n(\Z/q\Z)$ --- a manifestation of ``strong approximation'' --- so that 
\begin{align} 
[\SL_n(\Z):\Gamma(q)] &= \#\SL_n(\Z/q\Z) \\
&= \prod_{\substack{p \in \bP \\ p \mid q}} \# \SL_n(\Z/p\Z) \\
&= \prod_{\substack{p \in \bP \\ p \mid q}} \frac{(p^n-1)(p^n-p)(p^n-p^2) \cdots (p^n-p^{n-1})}{p-1} \label{denn}
\end{align}
where the second equality follows from the Chinese remainder theorem.

Next, the number of orbits is precisely 
\beq \#\{ g \in \SL_n(\Z/q\Z) \, : \, g_{i,j} = 0 \}. \eeq
By the Chinese remainder theorem we simply need to compute, for $p$ prime, 
\beq \#\{ g \in \SL_n(\F_p) \, : \,  g_{i,j} = 0 \}. \eeq
We estimate the number of such matrices in $\mathrm{GL}_n(\F_p)$ and then divide by $\#\F_p^\times=p-1$ to get the desired count.
We have $p^{n-1}-1$ choices for the $j${th} column, then $p^n-p$ choices for another column to be linearly independent from that column, $p^n-p^2$ choices for a third column to be linearly independent from the span of those two columns, etc.

Therefore 
\beq \#\{ g \in \SL_n(\F_p) \, : \, g_{i,j} = 0 \} = \frac{(p^{n-1}-1)(p^n-p)(p^n-p^2) \cdots (p^n-p^{n-1})}{p-1} \eeq
and
\beq \label{numn} \#\{ g \in \SL_n(\Z/q\Z) \, : \, g_{i,j} = 0 \} = \prod_{\substack{p \in \bP \\ p \mid q}} \frac{(p^{n-1}-1)(p^n-p)(p^n-p^2) \cdots (p^n-p^{n-1})}{p-1}.  \eeq

It thus follows from \eqref{denn} and \eqref{numn} that 
\begin{align}
\#\{ g \in V_T(\Z) \, : \, q \mid g_{i,j}  \} &= \frac{\#\{ g \in \SL_n(\Z/q\Z) \, : \, g_{i,j} = 0 \}}{[\SL_n(\Z):\Gamma(q)]} \#V_T(\Z) + O\left(\prod_{\substack{p \in \bP \\ p \mid q}} p^{n^2-2} \#V_T(\Z)^{1-\delta} \right) \\
&= \prod_{\substack{p \in \bP \\ p \mid q}} \frac{p^{n-1}-1}{p^n-1} \#V_T(\Z) + O\left(q^{n^2-2} \#V_T(\Z)^{1-\delta} \right),
\end{align}
as claimed.
\end{proof}

\section{Sieving}

\subsection{The Granville--Soundararajan framework} \label{GSsieve}

For a multiset $A = \{a_1, \ldots, a_x \}$ and a positive integer $d$, define \beq \A_d = \#\{n \le x \, : \, d \mid a_n \}. \eeq
Suppose that, for square-free $d$, $\A_d$ can be written in the form $\frac {h(d)}d x + r_d$ for some non-negative multiplicative function $h$ with $0 \le h(d) \le d$ for all square-free $d$.
For a set of primes $\P$ and $a \in \A$, define $\omega_\P(a) = \#\{ p \in \P \, : \, p \mid a \}$.
Define $\mu_\P = \sum_{p \in \P} \frac{h(p)}{p}$ and $\sigma_\P^2 = \sum_{p \in \P} \frac{h(p)}{p}\left(1 - \frac{h(p)}{p} \right)$.
Define $\mathcal{D}_k(\P)$ to be the set of square-free integers which are the products of at most $k$ elements of $\P$.
Finally, define $C_k = \dfrac{\Gamma(k+1)}{2^{k/2} \Gamma\left(\frac k2 +1\right)}$.

The following is \cite[Proposition 3]{GranvilleSound}.

\begin{claim}[Granville--Soundararajan] \label{GranvSound} Uniformly for all natural numbers $k \le \sigma_\P^{2/3}$, we have 
\beq \sum_{a \in \A} (\omega_\P(a) - \mu_\P)^k = C_k x \sigma_\P^k \left( 1 + O\left( \frac {k^3}{\sigma_\P^2} \right) \right)+  O\left( \mu_\P^k \sum_{d \in \mathcal{D}_k(\P)} |r_d| \right) \eeq
if $k$ is even
and
\beq \sum_{a \in \A} (\omega_\P(a) - \mu_\P)^k \ll C_k x \sigma_\P^k \frac {k^{3/2}}{\sigma_\P} +  \mu_\P^k \sum_{d \in \mathcal{D}_k(\P)} |r_d|  \eeq
if $k$ is odd.
\end{claim}

\subsection{Proof of the main theorem}
In this section, we prove \autoref{eksln}.

Fix $n \ge 2$ and $(i,j) \in \{ 1, \ldots, n \}^2$.

Adopting the notation in \cite{GranvilleSound}, we define
\beq 
\mathcal{A}=\{g_{i,j} : g \in V_T(\Z)\},
\eeq
a multiset of size $x:=\#V_T(\Z)$.

Note that \beq \label{slnasymp} x \sim c_n T^{n^2-n} \eeq by a special case of equation (1.12) in \cite[Example 1.6]{DukeRudnickSarnak}, with the explicit formula for $c_n$ given by $c_n = \dfrac {\pi^{n^2/2}}{\Gamma\left(\frac n2 \right) \Gamma\left(\frac{n^2-n+2}{2} \right) \zeta(2) \dots \zeta(n)}$.

For our application, we take 
\beq \mathcal{P}=\{\text{all primes} \le T^{\epsilon(T)}\}
\eeq
where $\epsilon$ is a real function which tends to $0$ at infinity that we shall choose later. 
More precisely, we shall choose $\epsilon(T)$ to be of the form 
\[
\epsilon(T):=\frac{1}{(\log \log T)^\psi}
,\]
where $\psi$ is a fixed positive constant to be chosen later.

Finally, we define the multiplicative function $h \colon \Z_{\ge 1} \to \R_{\ge 0}$ by
\beq
h(q)=\mu^2(q) q \prod_{\substack{p \in \bP \\ p \mid q}} \frac{p^{n-1}-1}{p^n-1}
\eeq
In particular, for $p$ prime, \beq \frac {h(p)}p=\frac {p^{n-1}-1}{p^n-1} = \frac 1p + O\left(\frac{1}{p^n} \right). \eeq

Therefore, using Mertens' theorem we have 
\beq \label{mean}
\mu_{\mathcal{P}} = \llt + O(\log \log \log T)
\eeq
and likewise  
\beq \label{variance}
\sigma_{\mathcal{P}}^2 = \llt + O(\log \log \log T).
\eeq

Now \autoref{key_input} can be restated as follows: there exists a positive $\delta$ such that for every square-free positive integer $q$,
\beq \label{sieve} \mathcal{A}_q = \frac {h(q)}q x + r_q \eeq
where 
\beq \label{sieve_err} r_q = O(q^{n^2-2} x^{1-\delta}). \eeq

To help simplify the presentation, we introduce some probabilistic language.

For each $T > 0$, we equip the finite set $V_T(\Z)$ with the uniform probability measure.
For the given pair $(i, j) \in \{1, \ldots, n\}^2$, we define the random variable $\omega_{i,j}$ on $V_T(\Z)$ by $\omega_{i,j} \colon g \mapsto \omega(g_{i,j})$.
We can thus restate \autoref{eksln} as the convergence in distribution, as $T \to +\infty$, of the random variables $\frac{\omega_{i,j}(g) - \llt}{\sqrt{\llt}}$ to the standard Gaussian $\mathcal{N}(0,1)$.

The estimates for the moments in \autoref{GranvSound} apply to $\omega_\P$ which is a truncated version of the random variable we are interested in; still, choosing $\epsilon$ carefully allows us to extract all the information we need, as we make clear in the following lemma.

\begin{lemma} \label{trunc_equiv}
As $T \to +\infty$, the random variables $\frac{\omega_{i,j} - \mu_\P}{\sigma_P}$ converge in distribution to $\mathcal{N}(0,1)$ if and only if the random variables $\frac{\omega_\P - \mu_\P}{\sigma_P}$ do.
\end{lemma}

\begin{proof}
We have 
\[
c\in \Z\setminus \{0\}, z>1
\Rightarrow 
\#\{p \mid c: p>z\}
\leq \frac{\log |c|}{\log z}
,\]
where the primes are counted with multiplicity.
Therefore if $0<|c|\leq T$ and choosing $z = T^{\epsilon(T)}$, then 
\[
\omega(c) - \omega_{\mathcal{P}}(c)
\le \frac{\log T}{\log T^{\epsilon(T)}} 
\le \frac{1}{\epsilon(T)}
.\]
Observe that if we can choose the function
$\epsilon$ such that 
\beq \label{epsilon}
\frac{1}{\epsilon(T)}
=o(\sqrt{\log \log T}),
\eeq
then 
by the first sentence of \cite[Remark 1]{BillingsleyAMM}, we are done.
We can simply pick any $\psi \in (0, \frac 12)$ to have our truncation function $\epsilon$ tend to $0$ at infinity and satisfy \eqref{epsilon}, which completes the proof.
\end{proof}

Since we are interested in proving a central limit theorem, we shall not worry about the uniformity in $k$ --- it is, in principle, possible to keep track of that and obtain good estimates for all moments. 

In particular, \autoref{GranvSound} reads, as $T$ (or equivalently $x$) goes to infinity:
for every even $k$, 
\beq \label{even_moments} \sum_{a \in \A} \left(\frac{\omega_\P(a) - \mu_\P}{\sigma_\P} \right)^k = C_k x \left(1 + O\left(\frac 1{\sigma_\P^2} \right) \right) +  O\left(\frac{\mu_\P^k}{\sigma_\P^k} \sum_{q \in \mathcal{D}_k(\P)} |r_q| \right) \eeq
and for every odd $k$,
\beq \label{odd_moments} \sum_{a \in \A} \left(\frac{\omega_\P(a) - \mu_\P}{\sigma_\P} \right)^k \ll  \frac x {\sigma_\P} +  \frac{\mu_\P^k}{\sigma_\P^k} \sum_{q \in \mathcal{D}_k(\P)} |r_q|.  \eeq
To conclude using that result, we first need to show that the error term
\[
\mathcal{R}_k(T)
:=
\sum_{q \in \mathcal{D}_k(\P)} |r_q|
\]
is sufficiently small.
To do so, we first observe that every $q$ in $\mathcal{D}_k(\P)$
satisfies 
\[
q\leq (\max\{p:p \in \mathcal{P}\})^k
\leq T^{k \epsilon(T)}
.\]
Using this with \eqref{sieve_err}
we obtain
\[\mathcal{R}_k(T)
\ll x^{1-\delta}
\sum_{q\leq  T^{k \epsilon(T)}} q^{n^2-2}
\leq 
 T^{(n^2-1) k \epsilon(T)}
x^{1-\delta}
\ll 
x^{1-\delta + \frac {n^2-1}{n^2-n} k \epsilon(\gamma_n x^{1/(n^2-n)})}
,\]
where in the last inequality we recalled \eqref{slnasymp} and let $\gamma_n = \left(\frac{1}{c_n}\right)^{\frac{1}{n^2-n}}$.

Since $\epsilon(x)=o(1)$ we obtain that 
\[\mathcal{R}_k(T)
\ll
x^{1-\frac{\delta}{2}}
.\]
Injecting the above estimate 
into \eqref{even_moments} and \eqref{odd_moments} respectively,
we obtain that we have --- recalling the main terms of \eqref{mean} and \eqref{variance} and noting that as $T \to +\infty$, $\llt \sim \log \log x$ due to \eqref{slnasymp} --- for every even $k$,
\[
\frac 1x \sum_{a \in \mathcal{A}}  \left(\frac{\omega_\P(a) - \mu_\P}{\sigma_\P} \right)^k
=C_k \left(1+O\left(\frac{1}{\log \log x} \right) \right)
+
O\left((\log \log x)^{k/2} x^{-\frac{\delta}{2}} \right)
\]
and for every odd $k$,
\[
\frac 1x \sum_{a \in \A} \left(\frac{\omega_\P(a) - \mu_\P}{\sigma_\P} \right)^k \ll \frac 1{\sqrt{\log \log x}} + (\log \log x)^{k/2} x^{-\frac{\delta}{2}}
\]
where the implied constants depend on $k$
and we recall that $C_k$ is the $k$th moment of the standard normal distribution.
Equivalently, as $T \to +\infty$, we have that for every even $k$, 
\beq \frac 1{\#V_T(\Z)} \sum_{a \in \mathcal{A}}  \left(\frac{\omega_\P(a) - \mu_\P}{\sigma_\P} \right)^k \to C_k \eeq
and for every odd $k$,
\beq \frac 1{\#V_T(\Z)} \sum_{a \in \mathcal{A}}  \left(\frac{\omega_\P(a) - \mu_\P}{\sigma_\P} \right)^k \to 0, \eeq
which shows (since the normal distribution is characterised by its moments) that as $T \to +\infty$, the random variables $\frac{\omega_\P - \mu_\P}{\sigma_\P}$ converge in distribution to $\mathcal{N}(0,1)$.

Thanks to \autoref{trunc_equiv}, this means that the random variables  $\frac{\omega_{i,j} - \mu_\P}{\sigma_\P}$ converge in distribution to $\mathcal{N}(0,1)$. This is almost what we want, with the following lemma making the final connection and concluding the proof of \autoref{eksln}.

\begin{lemma}
As $T \to +\infty$, the random variables $\frac{\omega_{i,j} - \mu_\P}{\sigma_\P}$ converge in distribution to $\mathcal{N}(0,1)$ if and only if the random variables $\frac{\omega_{i,j} - \llt}{\sqrt{\llt}}$ do.
\end{lemma}

\begin{proof}
We write 
\beq
\frac{\omega_{i,j} - \llt}{\sqrt{\llt}} = \frac{\sigma_\P}{\sqrt{\llt}} \left( \frac{\omega_{i,j} - \mu_\P}{\sigma_\P} + \frac{\mu_\P - \llt}{\sigma_\P} \right)
\eeq
and note that, as $T \to +\infty$, $\frac{\sigma_\P}{\sqrt{\llt}} \to 1$, while $\frac{\mu_\P - \llt}{\sigma_\P} \to 0$, thanks to \eqref{variance} and \eqref{mean}.
The claim now follows from the third sentence in \cite[Remark 1]{BillingsleyAMM}.
\end{proof}

\bibliographystyle{plain}
\bibliography{bibliography}

\end{document}